\documentclass[11pt]{article}

\usepackage[T1]{fontenc}      
\usepackage{amsmath}
\usepackage{amsfonts}
\usepackage{amssymb}
\usepackage{natbib}
\usepackage{theorem}
\usepackage{geometry}

    \usepackage{appendix}

\newtheorem{Assumption}{{\bf Assumption}}[section]
\newtheorem{lemma}{Lemma}[section]
\newtheorem{proposition}{Proposition}[section]
\newtheorem{Remark}{{\bf Remark}}[section]

\newtheorem{definition}{{\bf Definition}}[section]

\newtheorem{theorem}{Theorem}[section]
\newtheorem{corollary}{Corollary}[section]

\newcommand{\tun}{\tau_1}
\newcommand{\td}{\tau_2}
\newcommand{\tdn}{\tau_{2n}}
\newcommand{\ttil}{\tilde{\tau}}
\newcommand{\E}{\mathbb{E}}
\newcommand{\fex}{f_1^{2n+1}}
\newcommand{\xiex}{\xi^{2n+1}}

\newcommand{\R}{\mathbb{R}}
\newcommand{\N}{\mathbb{N}}
\newcommand{\I}{\mathbb{I}}
\newcommand{\cf}{\mathcal{F}}
\newcommand{\ce}{\mathcal{E}}

\newcommand{\ct}{\mathcal{T}}
\newcommand{\stopt}{\mathcal{T}^d_{k,T}}
\newcommand{\stopo}{\mathcal{T}^d_{0,T}}
\newcommand{\F}{\mathbb{F}}

\newcommand{\secondhyp}{(A2) }
\newcommand{\thirdhyp}{(A3) }

\long\def\symbolfootnote[#1]#2{\begingroup\def\thefootnote{\fnsymbol{footnote}}
\footnote[#1]{#2}\endgroup}

\def\esssup{\text{ess sup}} 

\newcommand{\limn}{\lim_{n\rightarrow\infty}}

\newcommand{\ld}{\{0,\ldots, T\}}

\font\QEDlogofont=msam10 at 10pt
\def\QEDlogo{\hbox{\QEDlogofont\char'003}}
\def\QEDblogo{\hbox{\QEDlogofont\char'004}}
\newif\ifnologo\nologofalse
\newif\iflogo
\newif\ifblogo\blogofalse
\newif\iftopprhead\topprheadfalse

\def\prooffont{\normalsize}
\newenvironment{proof}{\par\addvspace{6pt plus2pt}
\par
\noindent\prooffont{\bf Proof\,:}\hskip6pt\ignorespaces}{%
   \ifblogo\hskip1.2pt
            \QEDblogo
   \else
   \ifnologo
   \else
   \hfill
            \QEDlogo
   \fi\fi
\par\addvspace{6pt plus2pt}\global\topprheadfalse}%
\begin{document}
\title{Optimal stopping and a non-zero-sum Dynkin game in discrete time with risk measures  induced by BSDEs}

\author{Miryana Grigorova\footnote{(Corresponding Author) Institute for Mathematics, Humboldt University-Berlin} \and Marie-Claire Quenez\footnote{Laboratoire de Probabilités et Modèles Aléatoires, Université Paris-Diderot}}

\date{}


\maketitle
 This is a preprint of an article whose final and definitive  form has been published  in Stochastics: An International Journal of Probability and Stochastic Processes, Volume 89, Issue 1, Pages 259-279 \copyright 
[2017] [copyright: Taylor \& Francis];  Stochastics is available online at: 
http://www.tandfonline.com/doi/full/10.1080/17442508.2016.1166505\\

\begin{center}
First version: 21 Aug 2015\\
Accepted: 13 Mar 2016\\
Published online: 01 Apr 2016
\end{center}
\begin{abstract}
We first study an optimal stopping problem in which a player (an agent)  uses a discrete stopping time in order to stop optimally a payoff process whose risk is  evaluated by a (non-linear) $g$-expectation. We then consider a non-zero-sum game on discrete stopping times  with two agents who aim at minimizing their respective risks.  The payoffs of the agents are assessed by $g$-expectations (with possibly different drivers for the different players). By using the results of the first part, combined with some ideas of S. Hamadène and J. Zhang, we construct a Nash equilibrium point of this game by a recursive procedure. Our results are obtained in the case of a standard Lipschitz driver $g$ without any additional assumption on the driver besides that ensuring the monotonicity of the corresponding $g$-expectation. 

\end{abstract}
%
%

\section{Introduction}
Initiated by Bismut \cite{Bismut1}, \cite{Bismut2} (in the linear case), the theory of backward stochastic differential equations (BSDEs for short) has been further developed by Pardoux and Peng \cite{Pape90} in their seminal paper. The theory of BSDEs has found a number of  applications in finance, among which pricing and hedging of European options, recursive utilities, risk measurement. BSDEs induce a  family of operators, the so-called $g$-conditional expectations, which have proved useful in  the literature on (non-linear) dynamic risk measures (cf., e.g.,  \cite{Pe04}, \cite{Gianin}). We recall that the $g$-conditional expectation at time $t\in[0,T]$ (where $T>0$ is a fixed time horizon, and $g$ is a Lipschitz driver) is the operator which maps a given terminal condition $\xi$ (where $\xi$ is a square-integrable random variable, measurable with respect to the information at time $T$) to the position at time $t$ of (the first component of) the solution of the BSDE with parameters $(g,\xi)$. This operator is denoted by $\mathcal{E}^g_{t,T}(\cdot).$ The operator $\mathcal{E}^g_{0,T}(\cdot)$ is called  $g$-expectation.\\ 
On the other hand, zero-sum Dynkin games have been introduced by Dynkin in  \cite{Dynkin} in the discrete-time framework. Since then, there have been lots of contributions to zero-sum Dynkin games both 
in discrete time and in continuous time (cf., e.g., \cite{Bismuth}, \cite{Neveu},  \cite{Alario}, \cite{Maingue}). A prominent financial example is given by the pricing problem of game options (also known as Israeli options), introduced by Kifer in \cite{Kifer}. 
Compared to the zero-sum case, there have been fewer works on the non-zero-sum case.  We can quote
\cite{Hamadene},  \cite{Laraki}, \cite{Hamadene2} in the continuous-time setting,  and
\cite{Morimoto}, \cite{Ohtsubo}, \cite{Solan1}, \cite{Hamadene3} in the discrete-time setting.
For a recent survey on zero-sum and non-zero-sum Dynkin games the reader is referred to \cite{Kifer-survey}. \\
In all the above references the players' payoffs are assessed by "classical" mathematical expectations. In the recent years, some authors 
(cf. \cite{DQS2}, and \cite{Bayraktar3}) have considered "generalized" Dynkin games in continuous time where 
the "classical" expectations are replaced by 
more general (non-linear) functionals. All these extensions  are limited to the zero-sum case.\\

In the present paper, we address a game problem with two "stoppers" whose profits (or payoffs) are assessed by non-linear dynamic risk measures and who aim at minimizing their risk. 
More concretely, the following situation is of interest to us:
we are given two  adapted processes $X$ and $Y$ with 
$X \leq Y$  and $X_T= Y_T$ a.s. 
We consider a   game option in discrete time, that is,  a contract between two "stoppers" (a seller and a buyer) who can act only
 at given times  $0=t_0< t_1 <\ldots< t_n=T$, where $n \in \N$. 
 The two agents  can thus choose their strategies only among the  discrete stopping times with values in the grid  $\{t_0, t_1, \ldots, t_n\}$. 
 We denote by $\ct^d_{0,T}$ this set of stopping times.
Recall that the game option gives the buyer the right  \textit{to exercise} at any (discrete) stopping  time 
$\tau_1 \in  \ct^d_{0,T}$ and the seller the right  \textit{to cancel} at any (discrete) stopping time  $\tau_2 \in \ct^d_{0,T}$. 
In financial terms, we could say that both the seller's cancellation strategy and the buyer's exercise strategy are of \emph{Bermudan type}.
 If the buyer exercises at time $\tau_1$ before the seller cancels, then the seller pays to the buyer the amount $ X_{\tau_1}$; otherwise, the buyer receives from the seller  
 the amount $ Y_{\tau_2}$ at the cancellation time $\tau_2$. The difference 
$Y_{\tau_2} - X_{\tau_2} \geq 0$ is interpreted as a penalty which the seller pays to the buyer in the case of an early cancellation of the contract. 
To summarize, if the seller chooses a cancellation time $\tau_2$ and the buyer chooses an exercise time $\tau_1$,  the former pays to the latter 
the payoff 
$$\, I(\tun, \td):=X_{\tun}\I_{\{\tun\leq \td\}} +Y_{\td}\I_{\{\td< \tun\}}$$
at time 
$\tau_1 \wedge \tau_2.$
 The seller's payoff at time 
$\tau_1 \wedge \tau_2$ is  equal to $-I(\tun, \td)$.

%
We emphasize that our aim here  is not to determine a "fair price" (or a "fair premium") for the game option, but rather to
determine an "equilibrium" for the game problem related to the risk minimization of both the seller and the buyer.    

The seller and the buyer are assumed  to evaluate  the risk of their payoffs in a (possibly) different manner.  
 
 
 The dynamic risk measure $\rho^{f_1}$ (resp. $\rho^{f_2}$) of the buyer (resp. the seller) is induced by a BSDE with driver $f_1$ (resp.  $f_2$). Up to a minus sign, $\rho^{f_1}$ (resp. $\rho^{f_2}$) corresponds  to the family of $f_1$-conditional expectations  (resp. $f_2$-conditional expectations). 

   If, at time $0$,  the buyer chooses $\tau_1$ as exercise time and the seller chooses $\tau_2$ as cancellation time, 
 the buyer's (resp. seller's) risk at time $0$ is thus given by
 \begin{equation*}
 \begin{aligned}
 \rho^{f_1}_{0,\tun\wedge \td  }(I(\tun, \td) )=- \ce^{f_1}_{0,\tun\wedge \td}(I(\tun, \td) ) \quad \\
 {\rm resp.}\quad 
 \rho^{f_2}_{0,\tun\wedge \td}(-I(\tun, \td))= - \ce^{f_2}_{0,\tun\wedge \td}( -I(\tun, \td)).
 \end{aligned}
 \end{equation*}
The goal of each of the agents is to minimize his/her risk.  
We are interested in finding an "equilibrium" pair of discrete stopping times $(\tau_1^*, \tau_2^*)$ for this problem, 
that is, a pair $(\tau_1^*, \tau_2^*)\in \ct^d_{0,T}\times \ct^d_{0,T}$ such that the first agent's risk attains its minimum at  $\tau_1^*$ when the strategy of the second one is fixed at $\tau_2^*$, and  
the second agent's risk  attains its minimum at  $\tau_2^*$ when the strategy of the first one is fixed at
 $\tau_1^*$.  
 In other words, we are looking for a pair $(\tau_1^*, \tau_2^*)\in \ct^d_{0,T}\times \ct^d_{0,T}$ satisfying 
\begin{equation*}
 \begin{aligned}
\max_{\tau_1 \in  \ct^d_{0,T}} \ce^{f_1}_{0,\tun\wedge \td^*  }(I(\tun, \td^*) )= \ce^{f_1}_{0,\tun^*\wedge \td^*  }(I(\tun^*, \td^*) ) \\
\,\,\max_{\tau_2 \in  \ct^d_{0,T}}\ce^{f_2}_{0,\tun^*\wedge \td  }(-I(\tun^*, \td) )
=\ce^{f_2}_{0,\tun^*\wedge \td^*  }(-I(\tun^*, \td^*) ).
 \end{aligned}
 \end{equation*}
 In the terminology of game theory, the above game problem is of a  \textit{non-zero-sum} type, and  a pair  $(\tau_1^*, \tau_2^*)$ satisfying the above properties corresponds to a {\em Nash equilibrium point} of this
 non-zero-sum game. This game problem can be seen as a "generalized" non-zero-sum
 Dynkin game problem (the term "generalized" refers to the fact that our problem involves non-linear expectations instead of classical expectations). 
Note that  in the trivial case where $f_1=f_2=0$, our game 
reduces to a classical zero-sum Dynkin game (with classical expectations). Let us also mention that  
we can easily incorporate  in our framework the situation where the seller and/or the buyer apply their respective risk measures to their  \textit{net gains} (that is the payoff minus the initial price of the game option) instead of to their \textit{payoffs}. If the initial price of the option is given by  $x$ (where $x>0$), the buyer's (resp. seller's) net gain at time $\tau_1\wedge \tau_2$ is given by   $I(\tau_1,\tau_2)-x$ (resp. $x-I(\tau_1,\tau_2)$).

We show that there exists a Nash equilibrium point for the  non-zero-sum Dynkin game problem  described above by using a constructive approach similar to that of \cite{Hamadene}, \cite{Hamadene2}, and \cite{Hamadene3}.  

This  approach  requires some results on optimal stopping with one agent. We are thus led  to considering first the following family of problems:
\begin{equation}\label{osp2_intro}
V(t_k):= ess\inf_{\tau\in \ct^d_{t_k,T}} \rho^g_{t_k,\tau}(\xi_\tau)=  -ess \sup_{\tau\in \ct^d_{t_k,T}} \ce^{g}_{t_k,\tau}(\xi_\tau), \text{ for all } k\in{0,1,\ldots,n},
\end{equation}
where $\ct^d_{t_k,T}$ denotes the set of discrete stopping  times valued in $\{t_k, \ldots, t_n  \}$ and where $\xi$ is a given square-integrable adapted process. 
We characterize the sequence of random variables $(V(t_k))_{k \in \N}$ via a backward recursive construction.
We  also show that the stopping time 
$\tau^*:=\tau^*(t_k):= \inf \{ t \in \{t_k, \ldots, t_n  \}, \, V(t)= \xi_{t} \}$, which belongs to $\ct^d_{t_k,T}$, is optimal for   \eqref{osp2_intro} at time $t_k$. To prove our results, we use a generalization of the martingale approach to the case of $g$-conditional expectations in discrete time. Our results are established  without any additional assumption on the driver $g$ besides that ensuring the monotonicity of the corresponding $g$-conditional expectations. In particular, we do not make an assumption of "groundedness" on $g$  (that is, the assumption $g(t,0,0,0)=0$),  nor do we make an assumption of concavity/convexity on $g$, nor  an assumption of "independence from $y$" (that is, $g(t,y,z,k)=g(t,y',z,k)$, for all $y,y'\in\R$), which are  sometimes  made in the literature.\\
Optimal stopping problems with one agent whose payoff is assessed by a non-linear expectation have been largely studied. A continuous-time version of Problem  \eqref{osp2_intro} has been considered in \cite{EQ}, \cite{Quenez-Sulem} and \cite{Grigorova}. Related works include, but are not limited to, \cite{Bayraktar-2}, \cite{Bayraktar}.
The discrete-time version of Problem \eqref{osp2_intro} has been introduced by Krätschmer and Schoenmakers in \cite{Schoe}, Example 2.7, who address the problem  under stronger assumptions on the driver $g$ 
than those
made in the present paper.  In particular, the authors of \cite{Schoe} need the zero-one law for $g$-expectation, and the property $\ce^{g}_{t,T}(\xi)=\xi$, for all $t$, for all $\xi$ square-integrable ${\cal F}_t$-measurable.
These two properties do not hold for a general Lipschitz driver $g$.
 For this reason, 
the results from  \cite{Schoe} are not applicable in our framework; thus, we have been led to studying  Problem \eqref{osp2_intro} by using different techniques.

The remainder of the paper is organized as follows:
In Section \ref{defin_sect}, we set the framework and the notation. Section \ref{sect_optimal} is dedicated to the optimal stopping problem with one player. In Subsection \ref{subsect1}, we define the notion of $g$-(super)martingales in discrete time and we give some of their properties; in Subsection \ref{subsect2}, we  characterize the value function of our optimal stopping problem and we show the existence of optimal stopping times.  In Section \ref{sect_Dynkin}, we formulate our non-zero-sum Dynkin game with two players and we show the existence of a Nash equilibrium.  In Section \ref{sect_develop}, we briefly comment upon possible extensions of our results. The Appendix contains a useful property of $g$-expectations (Prop. \ref{zero-one law general}), along with some related remarks, as well as the proof of an easy result from Section \ref{sect_optimal}.
 

\section{The framework}\label{defin_sect}
Let $T$ be a fixed positive real number. Let $(E, \mathcal{K})$ be a measurable space 
equipped with a $\sigma$-finite positive measure $\nu$.
Let $(\Omega,  {\cal F}, P)$ be a (complete) probability space equipped with a  one-dimensional Brownian motion $W$ and with an independent  Poisson random measure $N(dt,de)$ with compensator $dt\otimes\nu(de)$. 
We denote by $\tilde N(dt,de)$ the compensated process, i.e. $\tilde N(dt,de):=\tilde N(dt,de)-dt\otimes\nu(de).$
Let $\F = \{{\cal F}_t \colon t\in[0,T]\}$ 
be  the (complete) natural filtration associated with $W$ and $N$.\\ 
 We use the following notation: 
\begin{itemize}
\item 
 $L^2({\cal F}_T)$  is the set of random variables which are  ${\cal F} 
_T$-measurable and square-integrable.
\item $L^2_\nu$ is the set of $\mathcal{K}$-measurable functions $\ell:  E \rightarrow \R$ such that  $\|\ell\|_\nu^2:= \int_{ E}  |\ell(e) |^2 \nu(de) <  \infty.$
For $\ell\in L^2_\nu$, $k\in L^2_\nu$, we define $\langle \ell, {k}\rangle_\nu:=\int_E \ell(e){k} (e) \nu(de)$.
\item    $H^{2,T}$ is the set of 
real-valued predictable processes $\phi$ such that\\
 $\| \phi\|^2_{H^{2,T}} := E \left[(\int_0 ^T |\phi_t| ^2 dt)\right] < \infty.$ 

\item $H_{\nu}^{2,T}$ is  the set of real-valued processes $l: (\omega,t,e)\in\Omega\times[0,T] \times  E\mapsto l_t(\omega, e)\in\R$ which are {\em predictable}, that is $({\cal P} \otimes {\cal K})$-measurable,
and such that $$\| l \|^2_{H_{\nu}^{2,T}} :=E\left[ \int_0 ^T \|l_t\|_{\nu}^2 \,dt   \right]< \infty.$$




\item   ${\cal S}^{2,T}$ is the set of real-valued  RCLL adapted
 processes $\varphi$ such that\\
${\varphi}^2_{{\cal S}^{2,T}} := E(\sup_{t\in[0,T]} |\varphi_t |^2) <  \infty.$



\end{itemize}

We recall the following terminology from BSDE theory. 

\begin{definition}(Lipschitz driver, standard data)
A function $g$ is said to be a {\em driver} if the following two conditions hold: 
\begin{itemize}
\item (measurability) 
$g: \Omega \times [0,T]\times \R^2  \times L^2_\nu  \rightarrow \R $\\
$(\omega, t,y,z;\ell) \mapsto  g(\omega, t, y, z,\ell) $
  is $ {\cal P} \otimes {\cal B}(\R^2)  \otimes {\cal B}(L^2_\nu)
- $ measurable,  where  ${\cal P}$  is the predictable $\sigma$-algebra on $\Omega \times [0,T]$, ${\cal B}(\R^2)$ is 
the Borel   $\sigma$-algebra on $\R^2$, and ${\cal B}(L^2_\nu ) $ is the Borel 
$\sigma$-algebra on $L^2_\nu $.
\item (integrability) $E \left[(\int_0 ^T |g(t,0,0,0)| ^2 dt)\right]<\infty$.
\end{itemize} 
A driver $g$ is called a {\em Lipschitz driver} (or a \emph{standard Lipschitz driver}) if moreover there exists a constant $ K \geq 0$ such that $dP \otimes dt$-a.e.\,, 
for each $(y_1, z_1, \ell_1)\in\R^2 \times L^2_\nu$, $(y_2, z_2, \ell_2)\in\R^2 \times L^2_\nu$, 
$$|g(\omega, t, y_1, z_1,\ell_1) - g(\omega, t, y_2, z_2,\ell_2)| \leq 
K (|y_1 - y_2| + |z_1 - z_2|+  \|{\ell}_1 - {\ell}_2 \|_\nu).$$
A pair $(g,\xi)$ such that $g$ is a Lipschitz driver and $\xi\in L^2(\Omega,\cf_T,P)$ is called a \emph{pair of standard data}, or a \emph{pair of standard parameters}.
\end{definition}
Let $(\xi, g)$ be a pair of standard data. The BSDE associated with Lipschitz driver $g$, terminal time $T$, and terminal condition $\xi$, is formulated as follows:
\begin{equation}\label{BSDEeq}
Y_t=\xi+\int_t^T g(s,Y_s,Z_s,k_s) ds - \int_{t}^T Z_sdW_s -\int_{t}^T\int_E  k_s(e) \tilde N(ds,de), \text{ for } t\in[0,T].
\end{equation}
 We recall that the above BSDE admits a unique solution triplet $(Y,Z,k)$ in the space ${\cal S}^{2,T} \times H^{2,T} \times 
H_{\nu}^{2,T} $. 
We denote by $\ce^g_{\cdot,T}(\xi) $ the first component of the solution of that BSDE 
(i.e. $(\ce^g_{t,T}(\xi))_{t\in[0,T]}$ is the family of $g$-conditional evaluations of $\xi$ in the vocabulary of S. Peng).\\
Recall also (cf., e.g., El Karoui et al.  \cite{ElKaroui}) that if the terminal time is given by a stopping time $\tau\in\ct_{0,T}$ and if $\xi$ is $\cf_\tau$-measurable,  
the solution  of the BSDE associated with  terminal time $\tau$,  terminal condition $\xi$ and Lipschitz driver $g$ is defined as the solution of the BSDE with (fixed) terminal time $T$, terminal condition $\xi$ and Lipschitz driver 
$g^{\tau}$ given by 
$$g^{\tau} (t,y,z,\ell):= g(t,y,z,\ell)\I_{\{t\leq \tau\}}.$$
The first component of this solution is thus equal to $(\ce^{g^\tau}_{t,T}(\xi))_{t\in[0,T]}$. In the sequel it is also denoted by 
$(\ce^g_{t,\tau}(\xi))_{t\in[0,T]}$.
We have $\ce^g_{t,\tau}(\xi)=\xi$ a.s. on the set $\{t\geq \tau\}$.


Recall that  $T$ is interpreted as the final time horizon.  
 For each $ T' \in [0,T]$ and $ \eta \in L^2({\cal F}_ {T'})$, we set 
  \begin{equation}\label{definition}
  \rho^g _{t, T'}( \eta)  :=  -{\cal E}^g_{t, T'}( \eta) , \,\, \,\,\,0\leq t \leq  T', 
  \end{equation}
If $ T'$ represents a given maturity and $ \eta$ a financial position at time $ T'$,
then $\rho^g _{t, T'}( \eta)$ is interpreted as the risk of $ \eta$ at time 
$t$. The functional $\rho^g :  ( \eta ,  T') \mapsto  \rho^g _{\cdot, T'}( \eta) $ thus represents 
 a {\em dynamic risk measure} induced by the BSDE with driver $g$.
In order to ensure the {\em monotonicity} property of 
$\rho^g$, that is, the monotonicity property with respect to the financial position, which is naturally required for risk measures, we assume from now on that the driver $g$ satisfies the following assumption (cf. \cite[Thm. 4.2, combined with Prop. 3.2]{Quenez-Sulem-0} and the references therein).

\begin{Assumption}\label{Royer} 
Assume that  $dP \otimes dt$-a.e.\, for each $(y,z, {\ell}_1,{\ell}_2)$ $\in$ $ \mathbb{R}^2 \times (L^2_{\nu})^2$,
$$g( t,y,z, {\ell}_1)- g(t,y,z, {\ell}_2) \geq \langle \theta_t^{y,z, {\ell}_1,{\ell}_2}  \,,\,{\ell}_1 - {\ell}_2 \rangle_\nu,$$ 
where the mapping
\begin{equation*}
\theta:  [0,T]  \times \Omega\times\mathbb{R}^2 \times  (L^2_{\nu})^2  \rightarrow  L^2_{\nu}\,; \, (\omega, t, y,z, {\ell}_1,{\ell}_2)\mapsto 
\theta_t^{y,z, {\ell}_1,{\ell}_2}(\omega,\cdot)
\end{equation*}
 is ${\cal P } \otimes {\cal B}(\R^2) \otimes  {\cal B}( (L^2_{\nu})^2 )$-measurable,  and satisfies $ dP\otimes dt\otimes d\nu(e)$-a.e.\,, for each $(y,z, {\ell}_1,{\ell}_2)$ $\in$ $\mathbb{R}^2 \times (L^2_{\nu})^2$,
 \begin{equation}\label{condi}
\theta_t^{y,z, {\ell}_1,{\ell}_2} (e)\geq -1.
\end{equation}
Assume, moreover, that $\theta$ is uniformly bounded, in the sense that, 
 $ dP\otimes dt$-a.e.\,, for each $(y,z, {\ell}_1,{\ell}_2)$, $\|\theta_t^{y,z, {\ell}_1,{\ell}_2}\|_\nu \leq K$, where $K$ is a positive constant.
\end{Assumption}

\section{Optimal stopping with $g$-expectations in discrete time}\label{sect_optimal}
Let $(\xi_t)_{t\in[0,T]}$ be a given $\F$-adapted square-integrable process  modelling an agent's  dynamic financial position. The agent is allowed to "stop" only at given times  $0=t_0< t_1 <\ldots< t_n=T$, where $n \in \N$.
The agent's risk is assessed through a dynamic risk measure $\rho^g$ induced by a BSDE with a given Lipschitz driver $g$; the dynamic risk measure $\rho^g$ corresponds (up to a minus sign) to the family of $g$-conditional expectations .  The agent's aim at time $0$ is to choose his/her strategy in such a way that the risk of his/her position from  "time $0$-perspective" be minimal.
The minimal risk at time $0$  is defined by
\begin{equation}\label{optimal_stopping_problem_0}
V(0):= \inf_{\tau\in\stopo} \rho^g_{0,\tau}(\xi_\tau)=  -\sup_{\tau\in\stopo} \ce^{g}_{0,\tau}(\xi_\tau).
\end{equation}
We have $V(0) = -V_0$, where
\begin{equation}\label{optimal_stopping_problem}
V_0:=\sup_{\tau\in\stopo} \ce^{g}_{0,\tau}(\xi_\tau).
\end{equation}
We are thus facing an optimal stopping problem  in discrete time with $g$-expectation.
Our purpose in this section is to compute or characterize   the minimal  risk measure $V(0)$
(or equivalently,  $V_0$)  
and  to study the question of the existence of   an 
optimal stopping time. \\
In order to simplify the notation, we suppose from now on  that the terminal time $T$ is in  $\N$ 
and that $t_k=k$, for all $k=1,\ldots,n$. In this case, the set  ${\cal T}_{t_k,T}^d$ corresponds to the set $\stopt$ of stopping times  whose values are almost surely in   the set $\{k,k+1,\ldots,T\}$.  We will use the notation  $\F^d$ for the filtration $(\cf_k)_{k\in\{0,1,\ldots,T\}}$. 

\subsection{Discrete-time $g$-(super)martingales}\label{subsect1}
We introduce the notion of  discrete-time $g$-(super)martingales, which is to be compared with the definition of a $\mathcal{E}^g$-supermartingale (in continuous time), respectively $\mathcal{E}^g$-martingale (in continuous time).
\begin{definition}\label{def_supermartingale}
Let $(\phi_k)_{k\in\{0,1,\ldots, T\}}$ be a sequence of square-integrable random variables, adapted to $(\cf_k)_{k\in\{0,1,\ldots, T\}}$.  We say that the sequence
$(\phi_k)$ is a \emph{$g$-supermartingale} (resp. \emph{$g$-martingale}) \emph{in discrete time} if 
$\phi_k\geq   \ce^{g}_{k,k+1}(\phi_{k+1}),$ for all $k\in\{0,1,\ldots, T-1\}.$
\end{definition}
\begin{Remark}
We note that if $(\phi_t)_{t\in[0,T]}$ is a $\mathcal{E}^g$-martingale (in continuous time) with respect to the filtration  $\F$, then $(\phi_k)_{k\in\{0,1,\ldots,T\}}$ is a $g$-martingale in discrete time  in the sense of Definition \ref{def_supermartingale}.   
\end{Remark}
\begin{Remark} In the case where $g\equiv 0$ (corresponding to the classical expectation), if $(\phi_k)_{k\in\{0,1,\ldots,T\}}$ is a \emph{discrete-time} martingale with respect to $(\cf_k)_{k\in\{0,1,\ldots, T\}}$, then $(\phi_k)_{k\in\{0,1,\ldots,T\}}$ can be extended into a \emph{continuous-time} martingale  with respect to $\F$ (with time parameter $t$ in $[0,T]$) by setting $\phi_t:= \phi_k$, for all $t\in(k,k+1)$, for all $k\in\{0,\ldots, T-1\}.$ This statement does not necessarily hold true in 
the case of a general driver $g$.    
\end{Remark}

\begin{Remark}\label{Rmk_g_expectation_extension}
 Let $(\phi_t)$ be a square-integrable adapted process. We recall that by definition  
$\ce^{g}_{t,s}(\phi_s)=\phi_s$, for all $T \geq t\geq s\geq 0$.\\
If $\phi_t$ is $\cf_t$-measurable, then  $\ce^{g^t}_{t,T}(\phi_t)=\ce^{g}_{t,t}(\phi_t)=\phi_t.$
\end{Remark}

\begin{theorem}\label{thm_martingale_arretee}
Let $(\phi_k)_{k\in\{0,1,\ldots, T\}}$ be a $g$-supermartingale (resp. a $g$-martingale) in discrete time. Let $\tau\in\stopo$. Then, the stopped process $(\phi_{k\wedge\tau})_{k\in\{0,1,\ldots, T\}}$ is a $g^\tau$-supermartingale (resp. a $g^\tau$-martingale) in discrete time. 
\end{theorem}
\begin{proof}
Let $k\in\{0,1,\ldots, T-1\}$.  Since $\ce^{g^\tau}_{k,k+1}(\phi_{(k+1)\wedge\tau}) = \ce^{g}_{k,(k+1)\wedge\tau}(\phi_{(k+1)\wedge\tau})$, 
it is sufficient to prove the following:
\begin{equation}\label{eq22_prop_optional_sampling}
\ce^{g}_{k,(k+1)\wedge\tau}(\phi_{(k+1)\wedge\tau})\leq \phi_{k\wedge\tau}.
\end{equation}
We write
\begin{equation}\label{eq_decomposition_prop_optional_sampling}
\ce^{g}_{k,(k+1)\wedge\tau}(\phi_{(k+1)\wedge\tau})=\I_{\{\tau\leq k\}}\ce^{g}_{k,(k+1)\wedge\tau}(\phi_{(k+1)\wedge\tau})+ \I_{\{\tau\geq k+1\}}\ce^{g}_{k,(k+1)\wedge\tau}(\phi_{(k+1)\wedge\tau}).
\end{equation}
Due to the definition of the solution of a standard BSDE with a stopping time as a terminal time, we have
\begin{equation}\label{eq_first_term_prop_optional_sampling}
\begin{aligned}
\I_{\{\tau\leq k\}}\ce^{g}_{k,(k+1)\wedge\tau}(\phi_{(k+1)\wedge\tau})=\I_{\{\tau\leq k\}}\phi_{(k+1)\wedge\tau}=
\I_{\{\tau\leq k\}}\phi_{\tau}.
\end{aligned}
\end{equation}
For the second term on the right-hand side of equation \eqref{eq_decomposition_prop_optional_sampling} we have 
\begin{equation}\label{eq_second_term_prop_optional_sampling}
\I_{\{\tau\geq k+1\}}\ce^{g}_{k,(k+1)\wedge\tau}(\phi_{(k+1)\wedge\tau})\leq 
\I_{\{\tau\geq k+1\}}\phi_{k\wedge\tau}.
\end{equation}
Indeed, after noticing that $\I_{\{\tau\geq k+1\}}$ is $\cf_k$-measurable, we apply Proposition \ref{zero-one law general} to obtain
\begin{equation}\label{eq_final_prop_optional_sampling}
\begin{aligned}
\I_{\{\tau\geq k+1\}}\ce^{g}_{k,(k+1)\wedge\tau}(\phi_{(k+1)\wedge\tau})&=\I_{\{\tau\geq k+1\}}\ce^{g^{(k+1)\wedge\tau}}_{k,T}(\phi_{(k+1)\wedge\tau})\\
&=\ce^{g^{(k+1)\wedge\tau}\I_{\{\tau\geq k+1\}}}_{k,T}(\I_{\{\tau\geq k+1\}}\phi_{k+1}).\\
\end{aligned}
\end{equation}
By using the fact that $g^{(k+1)\wedge\tau}\I_{\{\tau\geq k+1\}}=g^{k+1}\I_{\{\tau\geq k+1\}}$ and by applying Proposition \ref{zero-one law general} again, we get 
\begin{equation}\label{eq_final_supp_prop_optional_sampling}
\begin{aligned}
\ce^{g^{(k+1)\wedge\tau}\I_{\{\tau\geq k+1\}}}_{k,T}(\I_{\{\tau\geq k+1\}}\phi_{k+1})&=\ce^{g^{k+1}\I_{\{\tau\geq k+1\}}}_{k,T}(\I_{\{\tau\geq k+1\}}\phi_{k+1})\\
&=\I_{\{\tau\geq k+1\}}\ce^{g^{k+1}}_{k,T}(\phi_{k+1})=I_{\{\tau\geq k+1\}}\ce^{g}_{k,k+1}(\phi_{k+1}).
\end{aligned}
\end{equation}

As $\phi$ is a $g$-supermartingale in discrete time, we have
$\I_{\{\tau\geq k+1\}}\ce^{g}_{k,k+1}(\phi_{k+1})\leq \I_{\{\tau\geq k+1\}}\phi_{k}=\I_{\{\tau\geq k+1\}}\phi_{k\wedge\tau},$ which proves the inequality \eqref{eq_second_term_prop_optional_sampling}.
From \eqref{eq_first_term_prop_optional_sampling} and \eqref{eq_second_term_prop_optional_sampling} we get the desired inequality \eqref{eq22_prop_optional_sampling}. The theorem is thus proved.
\end{proof}

\begin{Remark}
We know that a "classical" (super)martingale in discrete time, stopped at a stopping time $\tau$, is again a (super)martingale. A $g$-(super)martingale in discrete time, stopped at a  stopping time $\tau\in\stopo$, is generally  not a $g$-(super)martingale, but a  $g^\tau$-(super)martingale (in virtue of the previous Theorem \ref{thm_martingale_arretee}).  This is illustrated by the following example.
Let $g$ be a driver which does not depend on $y$, $z$, and $\ell$ (i.e. $g(\omega,t,y,\ell)\equiv g(\omega,t)$). Recall that in this case the solution $Y$ of the BSDE with driver $g$ and terminal condition $\xi$ is given explicitly by 
\begin{equation}\label{Rk_new_eq_0} 
 Y_t=\E(\int_t^T g(s) ds +\xi|\cf_t),\text{ for all } t\in[0,T].
 \end{equation} 
Assume that $g$ is positive.
Let $\phi$ be a $g$-martingale in discrete time and take $\tau\equiv k$, where $k\in\{0,\ldots, T-1\}$. By applying \eqref{Rk_new_eq_0} with $\xi:=\phi_k$, we obtain 
$$\ce^{g}_{k,T}(\phi_{(k+1)\wedge k})=\ce^{g}_{k,T}(\phi_{k})=\E(\int_k^T g(s) ds +\phi_k|\cf_k)=\E(\int_k^T g(s) ds|\cf_k)+\phi_k>\phi_k,$$
the inequality being due to the positivity of $g$. Hence, $\phi$ stopped at $k$ is not a $g$-martingale in discrete time.
\end{Remark}

We now establish an "optional sampling" result for $g$-supermartingales (resp. for $g$-martingales). The result can be obtained as a corollary of the previous theorem. 
\begin{corollary} \label{cor_optional_sampling}
Let $(\phi_k)_{k\in\{0,1,\ldots, T\}}$ be a $g$-supermartingale (resp. a $g$-martingale) in discrete time. Then, for  $\sigma, \tau$ in $\stopo$ such that $\sigma\leq\tau$ a.s., we have 
$$\ce^{g}_{\sigma,\tau}(\phi_{\tau})\leq \phi_{\sigma} \text{ (resp. }=\phi_{\sigma}) \text{ a.s.} $$
\end{corollary}

\begin{proof}
We prove the result for the case of a $g$-supermartingale; the case of a $g$-martingale can be treated similarly.
Let $\sigma, \tau$ in $\stopo$ be such that $\sigma\leq\tau$ a.s. We notice that it suffices to prove the following property:
\begin{equation}\label{eq1_prop_optional_sampling}
\ce^{g}_{k\wedge\tau,\tau}(\phi_{\tau})\leq \phi_{k\wedge\tau}, \text{ for all } k\in\{0,1,\ldots, T\}.
\end{equation}
Indeed, this property proven, we will have
\begin{equation*}
\begin{aligned}
\ce^{g}_{\sigma,\tau}(\phi_{\tau})&=\ce^{g}_{\sigma\wedge\tau,\tau}(\phi_{\tau})=\sum_{k=0}^T \I_{\{\sigma=k\}}\ce^{g}_{k\wedge\tau,\tau}(\phi_{\tau})\leq \sum_{k=0}^T \I_{\{\sigma=k\}}\phi_{k\wedge\tau}=\phi_{\sigma\wedge\tau}=\phi_\sigma,
\end{aligned}
\end{equation*} 
which will conclude the proof.
Let us now prove property \eqref{eq1_prop_optional_sampling}. We proceed by backward induction. 
At the final time $T$ we have
$$\ce^{g}_{T\wedge\tau,\tau}(\phi_{\tau})=\ce^{g}_{\tau,\tau}(\phi_{\tau})=\phi_\tau=\phi_{T\wedge\tau}.$$
We suppose that the property \eqref{eq1_prop_optional_sampling} holds true for $k+1.$
Then, by using this induction hypothesis, the time-consistency and  the monotonicity of the $g$-conditional expectation (the latter property holds under Assumption \ref{Royer}), we get
$$\ce^{g}_{k\wedge\tau,\tau}(\phi_{\tau})=\ce^{g}_{k\wedge\tau,(k+1)\wedge\tau}(\ce^{g}_{(k+1)\wedge\tau, \tau}(\phi_{\tau}))\leq \ce^{g}_{k\wedge\tau,(k+1)\wedge\tau}(\phi_{(k+1)\wedge\tau}).
$$
In order to conclude, it remains to prove 
\begin{equation}\label{eq2_prop_optional_sampling}
\ce^{g}_{k\wedge\tau,(k+1)\wedge\tau}(\phi_{(k+1)\wedge\tau})\leq \phi_{k\wedge\tau}.
\end{equation}
We have 
\begin{equation}\label{eq_number_74}
  \I_{\{\tau\geq k\}}\ce^{g}_{k\wedge \tau,(k+1)\wedge\tau}(\phi_{(k+1)\wedge\tau})=\I_{\{\tau\geq k\}}\ce^{g}_{k,(k+1)\wedge\tau}(\phi_{(k+1)\wedge\tau})\leq 
\I_{\{\tau\geq k\}}\phi_{k\wedge\tau},
\end{equation}
where we have used Theorem \ref{thm_martingale_arretee} to obtain the inequality.

By Proposition \ref{zero-one law general}, we have 
\begin{equation}\label{eq_number_75}
\I_{\{\tau<k\}}\ce^{g}_{ k\wedge\tau,(k+1)\wedge\tau}(\phi_{(k+1)\wedge\tau})
=\ce^{g^{(k+1)\wedge\tau}\I_{\{\tau<k\}}}_{ \tau,T}(\phi_{\tau}\I_{\{\tau<k\}}).
\end{equation}
According to the convention given in Remark \ref{Rmk_measurability},
the "driver" $g^{(k+1)\wedge\tau}(s,y,z,\ell) \I_{\{\tau<k\}}$ is here equal to 
$g^{(k+1)\wedge\tau}(s,y,z,\ell) \I_{\{\tau<k\}}   \I_{]\tau,T]}(s)$, which is equal to zero. Hence, we have
$$\ce^{g^{(k+1)\wedge\tau}\I_{\{\tau<k\}}}_{ \tau,T}(\phi_{\tau}\I_{\{\tau<k\}})=\ce^{0}_{ \tau,T}(\phi_{\tau}\I_{\{\tau<k\}})= \phi_{\tau}\I_{\{\tau<k\}}.$$
 where the last equality is due to  the ${\cal F}_{\tau}$-measurability of 
$\phi_{\tau}\I_{\{\tau<k\}}$. \\
From equations \eqref{eq_number_74} and \eqref{eq_number_75} we deduce  \eqref{eq2_prop_optional_sampling}. 
The proposition is thus proved.
\end{proof}

\subsection{Discrete-time $g$-Snell envelope and optimal stopping times}\label{subsect2}
We now turn to the optimal stopping problem of the beginning of the section.
As is usual in optimal control, we embed the above optimization problem \eqref{optimal_stopping_problem} in a larger class of problems by considering 
\begin{equation}\label{optimal_stopping_problem_k}
V_k:=\esssup_{\tau\in\stopt} \ce^{g}_{k,\tau}(\xi_\tau), \text{ for } k\in\{0,1,\ldots, T\}.
\end{equation}


The following definition is analogous to the definition of the Snell envelope of a given process in discrete time, where we have replaced the  mathematical expectation of the classical setting by a $g$-expectation.
We define  the process $(U_k)_{k\in\{0,1,\ldots, T\}}$ by backward induction as follows:
\begin{equation}\label{Un}
\begin{cases}
U_T=\xi_T,\\
U_k=\max\big(\xi_k; \ce^{g}_{k,k+1}(U_{k+1})\big), \text{ for } k\in\{0,1,\ldots, T-1\}.
\end{cases}
\end{equation}
From \eqref{Un} we see by backward induction that $(U_k)$ is a well-defined, $(\cf_k)$-adapted sequence of square integrable random variables.  The sequence $(U_k)$ will be called the \emph{ $g$-Snell envelope in discrete time } of $(\xi_k).$

We now give a characterization of the $g$-Snell envelope in discrete time  
of $(\xi_k).$
\begin{proposition}\label{prop_smallest_supermartingale}
The sequence $(U_k)_{k\in\{0,1,\ldots, T\}}$ defined in equation \eqref{Un} is the smallest $g$-supermartingale in discrete time dominating the sequence $(\xi_k)_{k\in\{0,1,\ldots, T\}}$.  
\end{proposition}

The proof of the above proposition is similar to the  proof in the case of a classical 
expectation and is given in the Appendix for reader's convenience.

Let $k\in\{0,1,\ldots, T\}$ be given. 
We define the following stopping time:
\begin{equation}\label{nu_zero}
\nu_k:=\inf\{l\in\{k,\ldots, T\}:U_l=\xi_l\}.
\end{equation}
The following propositions hold true.
\begin{proposition}\label{prop_stopped_process}
Let $k\in\{0,1,\ldots, T-1\}.$ Let $\nu_k$ be the stopping time defined in \eqref{nu_zero}.
The sequence $(U_{l\wedge \nu_k})_{l\in\{k,\ldots, T\}}$ is a $g^{\nu_k}$-martingale in discrete time.
\end{proposition}

\begin{proof}
Let $l\in\{k,\ldots, T-1\}.$
We show that $U_{l\wedge\nu_k}=   \ce^{g}_{l,(l+1)\wedge\nu_k}(U_{(l+1)\wedge\nu_k})$ 
We write
\begin{equation}\label{eq_decomposition_prop_stopped_process}
\ce^{g}_{l,(l+1)\wedge\nu_k}(U_{(l+1)\wedge\nu_k})=\I_{\{\nu_k\leq l\}}\ce^{g}_{l,(l+1)\wedge\nu_k}(U_{(l+1)\wedge\nu_k})+ \I_{\{\nu_k\geq l+1\}}\ce^{g}_{l,(l+1)\wedge\nu_k}(U_{(l+1)\wedge\nu_k}).
\end{equation}
As in the proof of Theorem \ref{thm_martingale_arretee}, we have, by definition of the solution of the BSDE with a stopping time as a terminal time,  
\begin{equation}\label{eq_1_prop_stopped_process}
\I_{\{\nu_k\leq l\}}\ce^{g}_{l,(l+1)\wedge\nu_k}(U_{(l+1)\wedge\nu_k})=\I_{\{\nu_k\leq l\}} U_{(l+1)\wedge\nu_k}=\I_{\{\nu_k\leq l\}} U_{l\wedge\nu_k}.  
\end{equation}
For the second term on the right-hand side of equation \eqref{eq_decomposition_prop_stopped_process} we use again the same arguments as those of the proof of Theorem \ref{thm_martingale_arretee} (cf. equations \eqref{eq_final_prop_optional_sampling} and \eqref{eq_final_supp_prop_optional_sampling}) to show 
$$\I_{\{\nu_k\geq l+1\}}\ce^{g}_{l,(l+1)\wedge\nu_k}(U_{(l+1)\wedge\nu_k})=
\I_{\{\nu_k\geq l+1\}}\ce^{g}_{l,l+1}(U_{l+1}).        $$
From the definition of $\nu_k$ we see that $U_l>\xi_l$ on the set $\{\nu_k\geq l+1\}.$ Combining this observation with the definition of $U$ gives $U_l=\ce^{g}_{l,l+1}(U_{l+1})$ on  the set $\{\nu_k\geq l+1\}.$
Hence,
\begin{equation}\label{eq_2_prop_stopped_process}
\I_{\{\nu_k\geq l+1\}}\ce^{g}_{l,(l+1)\wedge\nu_k}(U_{(l+1)\wedge\nu_k})=
\I_{\{\nu_k\geq l+1\}}U_l=\I_{\{\nu_k\geq l+1\}}U_{l\wedge \nu_k}.
\end{equation}
Plugging \eqref{eq_1_prop_stopped_process} and \eqref{eq_2_prop_stopped_process} in \eqref{eq_decomposition_prop_stopped_process} gives the desired result.

\end{proof}

In the following proposition we show that the stopping time $\nu_k$ defined in \eqref{nu_zero} is optimal for the optimization problem \eqref{optimal_stopping_problem_k} at time $k$ and that the value function  $V_k$ of the problem is equal to $U_k$ (the  $g$-Snell envelope
 in discrete time  of $(\xi_k)$).
\begin{theorem}\label{prop_martingale_optimality_principle_general}
For $k\in\{0,1,\ldots, T\},$
\begin{equation}\label{eq_optimal_general}
U_k=\ce^{g}_{k,\nu_k}(\xi_{\nu_k})=\esssup_{\nu\in\ct_{k,T}} \ce^{g}_{k,\nu}(\xi_\nu),
\end{equation}
where $\nu_k:=\inf\{l\in\{k,\ldots, T\}:U_l=\xi_l\}.$
\end{theorem}

\begin{proof}
In the case where $k=T$ the result is trivially true. Suppose $k\in\{0,1,\ldots, T-1\}$. 
By using Proposition \ref{prop_stopped_process} and Corollary \ref{cor_optional_sampling} (applied with $\sigma=k$ and $\tau=T$), and the fact that $U_{\nu_k}=\xi_{\nu_k},$ we obtain
\begin{equation}\label{eq_A_prop_martingale_optimality_principle}
U_k=U_{k\wedge\nu_k}=\ce^{g^{\nu_k}}_{k,T}(U_{T\wedge\nu_k})=\ce^{g}_{k,\nu_k}(U_{\nu_k})=\ce^{g}_{k,\nu_k}(\xi_{\nu_k}).
\end{equation}
Let $\nu\in\ct_{k,T}.$ By using Proposition \ref{prop_smallest_supermartingale} and Corollary \ref{cor_optional_sampling} (applied with  $\sigma=k$ and $\tau=\nu$), 
as well as the monotonicity of the functional $\ce^{g}_{0,\nu}(\cdot)$, we get
\begin{equation}\label{eq_B_prop_martingale_optimality_principle}
U_k\geq\ce^{g}_{k,\nu}(U_{\nu})\geq \ce^{g}_{k,\nu}(\xi_{\nu}).
\end{equation} 
Combining equations \eqref{eq_A_prop_martingale_optimality_principle} and \eqref{eq_B_prop_martingale_optimality_principle} gives the desired conclusion.
 
\end{proof}


\begin{Remark} Families of non-linear operators  $\{\ce_{t} (\cdot): t \in [0,T] \}$ indexed by a single index (as opposed to doubly-indexed families) are considered in several papers in the literature on  optimal stopping with non-linear functionals (cf., e.g., \cite{Schoe}, and \cite{Bayraktar}).  
We note that we work here with the doubly-indexed family of operators 
$\{\ce^g_{t,t'} (\cdot): t,t'\in [0,T] \}$. This family of operators reduces to the family  
$\{\ce^g_{t,T} (\cdot): t \in [0,T] \}$, indexed by a single index, under the additional assumption  
$\ce^g_{t,t'} (\eta)= \ce^g_{t,T} (\eta)$, $0\leq t \leq t'$, for all $t' \in [0,T]$, for all $\eta$ $\in$ $L^2({\cal F}_{t'})$. By using the consistency property of $g$-conditional expectations, 
it can be shown that 
this assumption is equivalent to the assumption (mentioned in the introduction)
 $\ce^g_{t,T} (\eta)= \eta$, for all $t \in [0,T]$, for all $\eta$ $\in$ $L^2({\cal F}_{t})$.
\end{Remark}


\section{ A non-zero-sum Dynkin game  in discrete time related to risk minimization}\label{sect_Dynkin}
We now consider  a game problem  which is slightly more general than that of the introduction. 
We are given two agents $A^1$ and $A^2$ whose payoffs/financial positions   are defined via four  $\F$-adapted sequences $X^1$, $X^2$, $Y^1$, $Y^2$.

We make the following assumptions:
\begin{enumerate}
\item[(A1)]$X^1\leq Y^1$ and $X^2\leq Y^2$ (that is, $X^1_k\leq Y^1_k$ and $X^2_k\leq Y^2_k$, $\forall k\in\{0,\ldots,T\}$) 
\item[(A2)] $X_T^1=Y_T^1$ and $X_T^2=Y_T^2$.  
\item[(A3)] The processes $Y^1$, $Y^2$, $X^1$, $X^2$ satisfy $\E(\max_{k\in\{0,\ldots,T\}} |Y^1_k|^2)<\infty$,
$\E(\max_{k\in\{0,\ldots,T\}} |Y^2_k|^2)<\infty$, $\E(\max_{k\in\{0,\ldots,T\}} |X^1_k|^2)<\infty$,
$\E(\max_{k\in\{0,\ldots,T\}} |X^2_k|^2)<\infty$.

\end{enumerate}

The set of strategies of each of the agents at time $0$ is $\stopo$. We emphasize that both agents use discrete stopping times as strategies. 
If the first agent's strategy is $\tau_1\in\stopo$ and the second agent's strategy is $\tau_2\in\stopo$, the payoff  of the first (resp. second) agent at time $\tau_1\wedge\tau_2$ is given by: 
$$ X^1_{\tun}\I_{\{\tun\leq \td\}} +Y^1_{\td}\I_{\{\td< \tun\}} \quad(\text{resp. }  X^2_{\td}\I_{\{\td< \tun\}} +Y^2_{\tun}\I_{\{\tun\leq \td\}}),  $$
where we have adopted the following convention:  when $\tun=\td$,    it is the first player who is responsible for stopping the game. \\

The agents $A_1$ and $A_2$ evaluate  the risk of their respective payoffs in a (possibly) different manner.

More precisely, we are now given two standard Lipschitz drivers $f_1$ and $f_2$.
 
 The dynamic risk measure of the first agent is equal to $\rho^{f_1}=-{\cal E}^{f_1}$ and the dynamic risk measure of the second agent is equal to $\rho^{f_2}=-{\cal E}^{f_2}$.

 If the first agent's strategy is $\tau_1\in\stopo$ and the second agent's strategy is $\tau_2\in\stopo$, 
 the first agent's (resp. second agent's) risk at time $0$ is thus given by
 $ -J_1(\tun,\td)$ (resp. $ -J_2(\tun,\td)$) where
$$J_1(\tun,\td):= \ce^{f_1}_{0,\tun\wedge \td}(X^1_{\tun}\I_{\{\tun\leq \td\}} +Y^1_{\td}\I_{\{\td< \tun\}})$$
$$(\text{resp. }J_2(\tun,\td):= \ce^{f_2}_{0,\tun\wedge \td}( X^2_{\td}\I_{\{\td< \tun\}} +Y^2_{\tun}\I_{\{\tun\leq \td\}})).$$  
The two agents aim at  minimizing the risk of their payoffs.

The problem with the game option presented in the introduction can be seen as a particular case of the game described above, with 
the first agent $A_1$ corresponding to the buyer of the game option, the second agent $A_2$ corresponding to the seller, and with
$X^1= -Y^2= X$ and $Y^1=-X^2= Y$. Let us emphasize that even in this particular case  where the payoffs of the two agents are equal up to a minus sign, the game is of a non-zero-sum type due to the non-linearity of the dynamic risk measures. 
The situation, also mentioned in the introduction, where the seller and/or the buyer of the option apply their risk measures to  their \emph{net gains},   also enters in the above general framework. For instance, if the seller of the option takes into account  his/her \emph{net gain}, while the buyer considers \emph{the payoff of the option} only, we set: $X^1= X$, $Y^1= Y$, $Y^2= -X +x$, and $X^2= -Y+x$, where $x>0$ is the initial price of the option.  

In this section, we investigate the question of the existence of a Nash equilibrium point for the general game described above. 
\begin{definition}(Nash equilibrium point)
A pair of  stopping times $(\tau_1^*, \tau_2^*)\in\stopo\times\stopo$ is called a \emph{Nash equilibrium point}  
for the  above non-zero-sum Dynkin game if 
 $J_1(\tau_1^*, \tau_2^*)\geq J_1(\tun, \tau_2^*)$ and $J_2(\tau_1^*, \tau_2^*)\geq J_2(\tau_1^*, \tau_2),$ for any pair $(\tau_1, \tau_2)$ 
 of stopping times in $\stopo\times\stopo$. 
\end{definition}
In other words, a pair of strategies is a Nash equilibrium of the game if any unilateral deviation from that strategy on the part of one of the agents (the other agent's strategy remaining fixed) does not reduce his/her risk.
\noindent

\subsection{A preliminary result}
We begin by a preliminary proposition in which
we show that interchanging the strict and large inequalities in the expression of the payoff process does not change the corresponding value functions. This result will be used in the construction of a Nash equilibrium point in the following sub-section.
\begin{proposition}\label{equality_value_functions} 
Let $(X_k)$ and $(Y_k)$ be two $\F^d$-adapted sequences of square-integrable random variables such that $X_k\leq Y_k$, for all $k\in\{0,\ldots,T\}$, $X_T=Y_T$ and $\E(\max_{k\in\{0,\ldots,T\}} |X_k|^2)<\infty$.  Let $g$ be a standard driver and $\mu\in\stopo$.
For each $k\in\{0,1,\ldots, T\}$, let 
$$\bar\xi_k:= X_{k}\I_{\{k\leq \mu\}} +Y_{\mu}\I_{\{\mu< k\}} \quad {\rm and} \quad \xi_k:= X_{k}\I_{\{k< \mu\}} +Y_{\mu}\I_{\{\mu\leq k\}}.$$ 
For each $k\in\{0,1,\ldots, T\}$, let 
\begin{equation*}
\bar U_k:=\esssup_{\tau\in\stopt} \ce^{g}_{k,\tau \wedge \mu}(\bar\xi_\tau)\quad {\rm and} \quad  
 U_k:=\esssup_{\tau\in\stopt} \ce^{g}_{k,\tau \wedge \mu}(\xi_\tau),
\end{equation*}
which correspond to the $g^\mu$-Snell envelope in discrete time of $(\bar\xi_k)$ and $(\xi_k)$ respectively.

  Then, the following properties hold true: 
\begin{enumerate}
\item[(i)] $\bar U_k=Y_\mu=U_k$ a.s. on $\{\mu\leq k-1\}$.
\item[(ii)] $\bar U_k=U_k$ a.s. for all $k\in\{0,\ldots,T\}.$
\end{enumerate}
\end{proposition}
\begin{Remark}
 The result can be seen as an analogue in our framework of a result 
 by \cite{Maingue}(lemma 5) shown  in a continuous-time framework with right-continuous payoffs and classical expectation.
\end{Remark}

\begin{proof}
Let us prove $(i)$. We proceed by backward induction. A direct computation gives
$$U_T=\xi_T=Y_\mu  \quad {\rm and} \quad \bar U_T=\bar\xi_T=Y_\mu \quad \quad {\rm on} \quad \{\mu\leq T-1\}.$$
We suppose now that $\bar U_{k+1}=Y_\mu=U_{k+1}$ a.s. on $\{\mu\leq k\}$.\\
We show that  
$\bar U_k = Y_\mu=U_k$  a.s. on $\{\mu\leq k-1\}$.\\
By using the definition of $U_k$, we have
\begin{equation}\label{eq1_equality_value_functions}
U_k\I_{\{\mu \leq k-1\}}=\max\big(\xi_k \I_{\{\mu \leq k-1\}}; \I_{\{\mu \leq k-1\}}\ce^{g^\mu}_{k,k+1}(U_{k+1})\big).
\end{equation}
Now, by the definition of $\xi_k$,  we get 
\begin{equation}\label{eq1bis_equality_value_functions}
\xi_k \I_{\{\mu \leq k-1\}}= Y_\mu \I_{\{\mu \leq k-1\}}\quad 
\end{equation}
Moreover, by Proposition \ref{zero-one law general}  and the induction hypothesis, we get 
$$\I_{\{\mu \leq k-1\}}\ce^{g^\mu}_{k,k+1}(U_{k+1})=  \ce^{g^\mu\I_{\{\mu \leq k-1\}}}_{k,k+1}(U_{k+1}\I_{\{\mu \leq k-1\}})=\ce^{g^\mu\I_{\{\mu \leq k-1\}}}_{k,k+1}
(Y_\mu\I_{\{\mu \leq k-1\}})\quad $$
Now, the "driver" $g^\mu(\omega,s,y,z)\I_{\{\mu(\omega) \leq k-1\}}$ is equal to $g(\omega,s,y,z)\I_{\{s\leq \mu(\omega)\leq k-1\}}\I_{]k,T]}(s)$ (according to the convention used in Prop. \ref{zero-one law general}), which is equal to zero.  Moreover,  $Y_\mu\I_{\{\mu \leq k-1\}}$
is $\cf_k$-measurable. Therefore,
$\ce^{g^\mu\I_{\{\mu \leq k-1\}}}_{k,k+1}
(Y_\mu\I_{\{\mu \leq k-1\}}) =Y_\mu\I_{\{\mu \leq k-1\}}$  a.s.  
By combining this observation with equations \eqref{eq1_equality_value_functions} and \eqref{eq1bis_equality_value_functions}, we get
$U_k = Y_\mu$  a.s. on $\{\mu\leq k-1\}$.
By similar arguments we show that  
 $\bar U_k= Y_\mu$  a.s. on $\{\mu\leq k-1\}$ (to obtain this claim, it is sufficient to replace $U$ by $\bar U$, and $\xi$ by $\bar \xi$ in equation \eqref{eq1_equality_value_functions}). 
Property $(i)$ is thus proven.\\
Let us prove property $(ii)$. 
We proceed again by backward induction.
At the final time $T$ we have  $U_T=\xi_T=\bar\xi_T=\bar U_T$ (due to the assumption $X_T=Y_T$).
Suppose that $\bar U_{k+1}=U_{k+1}.$ 
Let us prove $\bar U_{k}=U_{k}.$ 
We note that $\xi_k=\bar\xi_k$ on the set $\{\mu=k\}^c.$ This observation, the induction hypothesis, and the definitions of $U_k$ and $\bar U_k$   lead to the equality 
$U_k=\bar U_k$ on the set $\{\mu=k\}^c.$ It remains to show that the equality also holds true on the set $\{\mu=k\}$.
By using the definition of $U_k$, the definition of $\xi_k$ and
Proposition \ref{zero-one law general}, 
we get
\begin{equation}\label{eq2_equality_value_functions}
\begin{aligned}
U_k\I_{\{\mu=k\}}&=\max\big(\xi_k \I_{\{\mu=k\}}; \I_{\{\mu=k\}}\ce^{g^\mu}_{k,k+1}(U_{k+1})\big)\\
&=
\max\big(Y_k \I_{\{\mu=k\}}; \ce^{g^\mu\I_{\{\mu=k\}}}_{k,k+1}(U_{k+1}\I_{\{\mu=k\}})\big).
\end{aligned}
\end{equation}
Now, by property $(i)$ which we have just proved, we have 
$$\ce^{g^\mu\I_{\{\mu=k\}}}_{k,k+1}(U_{k+1}\I_{\{\mu=k\}})=\ce^{g^\mu\I_{\{\mu=k\}}}_{k,k+1}(Y_k\I_{\{\mu=k\}}).
$$ 
As the "driver" $g^\mu\I_{\{\mu=k\}}$ is equal to $0$ and $Y_k\I_{\{\mu=k\}}$ is 
${\cal F}_k$-measurable,
$$\ce^{g^\mu\I_{\{\mu=k\}}}_{k,k+1}(Y_k\I_{\{\mu=k\}})
=\ce^{0}_{k,k+1}(Y_k\I_{\{\mu=k\}})=
Y_k\I_{\{\mu=k\}}.$$
From the previous three expressions, we obtain
$U_k\I_{\{\mu=k\}}=Y_k \I_{\{\mu=k\}}.$ 
Similarly, by using the definition of $\bar U_k$, the definition of $\bar \xi_k$, 
Proposition \ref{zero-one law general}, and property $(i)$,
we get 
\begin{equation*}
\begin{aligned}
\bar U_k\I_{\{\mu=k\}}&=\max\big(\bar \xi_k \I_{\{\mu=k\}}; \I_{\{\mu=k\}}\ce^{g^\mu}_{k,k+1}(\bar U_{k+1})\big)\\
&=
\max\big(X_k \I_{\{\mu=k\}}; \ce^{g^\mu\I_{\{\mu=k\}}}_{k,k+1}(\bar U_{k+1}\I_{\{\mu=k\}})\big)\\
&=\max\big(X_k \I_{\{\mu=k\}}; \ce^{g^\mu\I_{\{\mu=k\}}}_{k,k+1}(Y_k\I_{\{\mu=k\}})\big), 
\end{aligned}
\end{equation*}
Moreover, $\ce^{g^\mu\I_{\{\mu=k\}}}_{k,k+1}(Y_k\I_{\{\mu=k\}})=\ce^{0}_{k,k+1}(Y_k\I_{\{\mu=k\}}) 
=Y_k\I_{\{\mu=k\}}$.
Hence, $\bar U_k\I_{\{\mu=k\}}= \max (X_k; Y_k)  \I_{\{\mu=k\}}$.
As $X_k\leq Y_k$ by assumption, we obtain
$\bar U_k\I_{\{\mu=k\}}=Y_k\I_{\{\mu=k\}}.$ \\
Thus, 
the equality $\bar U_k\I_{\{\mu=k\}}=U_k\I_{\{\mu=k\}}$ holds,  which concludes the proof.
\end{proof}

\subsection{Construction of a Nash equilibrium point}
  
Following ideas of \cite{Hamadene}, and \cite{Hamadene2}, we  construct a Nash equilibrium point of the game described above by a recursive procedure. We also rely on the preliminary result of the previous subsection (Prop. \ref{equality_value_functions}) and on the results of Section \ref{sect_optimal}.
\begin{theorem}
Our non-zero-sum game with $g$-expectations in discrete time admits a Nash equilibrium point. 
\end{theorem}
To prove this theorem, 
we construct a pair $(\tau_{2n+1},\tau_{2n+2})_{n\in\N}$ of non-increasing sequences of stopping times and we show that the limit (as $n\rightarrow\infty$) is a NEP of the  game defined in Section \ref{defin_sect}.\\   
We set $\tun:=T$ and $\td:=T$. We suppose that the stopping times $\tau_{2n-1}$ and $\tau_{2n}$ have been defined.\\
We define $\fex(\omega,t,\cdot,\cdot, \cdot):=f^{\tau_{2n}}_1(\omega,t,\cdot,\cdot,\cdot):=f_1(\omega,t,\cdot,\cdot,\cdot)\I_{\{t\leq \tau_{2n}(\omega)\}}$. We note that $\fex$ is a standard Lipschitz  driver, as $f_1$ is a standard Lipschitz driver.\\ 
We set, for all $k\in\{0,\ldots, T\}$,
\begin{equation}\label{eq_definitions}
\begin{aligned}
\xiex_k&:= X^1_{k}\I_{\{k< \tdn\}} +Y^1_{\tdn}\I_{\{\tdn\leq k\}}  \\
W_k^{2n+1}&:=\esssup_{\tau\in\stopt}\ce_{k,\tau}^{\fex}(\xiex_\tau)\\ 
\ttil_{2n+1}&:=\inf\{k\in\{0,\ldots, T\}: W_k^{2n+1}=\xiex_k\}\\
\tau_{2n+1}&:=(\ttil_{2n+1}\wedge \tau_{2n-1})\I_{\{\ttil_{2n+1}\wedge \tau_{2n-1}<\tau_{2n}\}}+\tau_{2n-1}\I_{\{\ttil_{2n+1}\wedge \tau_{2n-1}\geq\tau_{2n}\}}.
\end{aligned}
\end{equation}
Due to Theorem \ref{prop_martingale_optimality_principle_general}, the stopping time $\ttil_{2n+1}$ is optimal for the above optimization problem \eqref{eq_definitions} at time $0$: more precisely, we have
$W_0^{2n+1}=\sup_{\tau\in\stopo}\ce_{0,\tau}^{\fex}(\xiex_\tau)= \ce_{0,\ttil_{2n+1}}^{\fex}(\xiex_{\ttil_{2n+1}})$. Moreover, 
thanks to Proposition \ref{equality_value_functions} 
 applied with $\mu:=\tau_{2n}$ and $g:=f_1$, we have 
\begin{equation}\label{equation_equality_values}
W_0^{2n+1}=\sup_{\tau\in\stopo} J_1(\tau, \tau_{2n}).
\end{equation}
We gather  some  more observations on the objects defined above in the following Remark \ref{The first remark} and Proposition  \ref{prop_properties_optimal_stopping_time}.

\begin{Remark}\label{The first remark}
\begin{enumerate}
\item[(i)]For all $n\in\N,$ $\tau_{2n+1}$ is a stopping time (this observation follows directly from the definition of $\tau_{2n+1}$). 
\item[(ii)]For all $n\in\N$, for all $\tau\in\stopo$, $\xiex_\tau=\xiex_{\tau\wedge\tdn}$.
\end{enumerate}
\end{Remark}
Moreover, we have the following proposition.
\begin{proposition}\label{prop_properties_optimal_stopping_time}
\begin{enumerate}
\item[(i)]$W_k^{2n+1}\I_{\{\tau_{2n}\leq k\}}=Y^1_{\tdn}\I_{\{\tau_{2n}\leq k\}}.$
\item[(ii)]For all $n\in\N,$ $\ttil_{2n+1}=\inf\{k\in\{0,\ldots, T\}:W_k^{2n+1}=X_k^1\}\wedge\tdn.$ In particular, $\ttil_{2n+1}\leq\tdn,$ for all $n\in\N$. 
\end{enumerate}
\end{proposition}
\begin{proof} Let us prove $(i).$ From the definition of $\xi^{2n+1}$ we have 
\begin{equation}\label{prop_eq_1}
\xi^{2n+1}_\tau \I_{\{\tau_{2n}\leq k\}}= Y^1_{\tau_{2n}} \I_{\{\tau_{2n}\leq k\}}, \text{ for all } \tau \in\stopt.
\end{equation} 
Due to the definition of the solution of a standard BSDE with a stopping time as a terminal time, we have
\begin{equation}\label{prop_eq_2}
\I_{\{\tau_{2n}\leq k\}}\ce^{\fex}_{k, \tau\wedge \tau_{2n}}(\xi^{2n+1}_\tau)=  \I_{\{\tau_{2n}\leq k\}}\xi^{2n+1}_\tau, \text{ for all } \tau \in\stopt. 
\end{equation} 
Thus, 
\begin{equation*}
\begin{aligned}
\I_{\{\tau_{2n}\leq k\}}W_k^{2n+1}&=\I_{\{\tau_{2n}\leq k\}}\esssup_{\tau\in\stopt}\ce^{\fex}_{k, \tau\wedge \tau_{2n}}(\xi^{2n+1}_\tau)\\
&= \esssup_{\tau\in\stopt}\I_{\{\tau_{2n}\leq k\}}\ce^{\fex}_{k, \tau\wedge \tau_{2n}}(\xi^{2n+1}_\tau)
= \I_{\{\tau_{2n}\leq k\}} \xi^{2n+1}_\tau= \I_{\{\tau_{2n}\leq k\}} Y^1_ {\tau_{2n}},
\end{aligned}
\end{equation*}
where we have used equation \eqref{prop_eq_2} to obtain the last but one equality, and equation \eqref{prop_eq_1} to obtain the last. \\
\noindent
Let us prove $(ii)$. 
By using the definition of $\ttil_{2n+1}$, and that of $\xiex$, we have
$$\ttil_{2n+1}=\inf\{k\in\ld:W_k^{2n+1}\I_{\{k<\tau_{2n}\}}+W_k^{2n+1}\I_{\{\tau_{2n}\leq k\}}=X_k^{1}\I_{\{k<\tau_{2n}\}}+Y^1_{\tdn}\I_{\{\tau_{2n}\leq k\}}\}.$$
Thanks to this observation and to the previous property $(i)$, we obtain
\begin{equation*}
\begin{aligned}
\ttil_{2n+1}&=\inf\{k\in\ld:W_k^{2n+1}\I_{\{k<\tau_{2n}\}}=X_k^{1}\I_{\{k<\tau_{2n}\}}\}\\
&=\inf\{k\in\ld:W_k^{2n+1}=X_k^1\}\wedge\tdn.
\end{aligned}
\end{equation*}
\end{proof}

We define $f_2^{2n+2}(\omega,t,\cdot,\cdot, \cdot):=f_2^{\tau_{2n+1}}(\omega,t,\cdot,\cdot,\cdot):=f_2(\omega,t,\cdot,\cdot,\cdot)\I_{\{t\leq \tau_{2n+1}(\omega)\}}$. Similarly to the definitions of \eqref{eq_definitions}, we set
\begin{equation}\label{eq_definitions_2}
\begin{aligned}
\xi_k^{2n+2}&:= X^2_{k}\I_{\{k< \tau_{2n+1}\}} +Y^2_{\tau_{2n+1}}\I_{\{\tau_{2n+1}\leq k\}}  \\
W_k^{2n+2}&:=\esssup_{\tau\in\stopt}\ce_{k,\tau}^{f_2^{2n+2}}(\xi^{2n+2}_\tau)\\ 
\ttil_{2n+2}&:=\inf\{k\in\ld: W_k^{2n+2}=\xi^{2n+2}_k\}\\
\tau_{2n+2}&:=(\ttil_{2n+2}\wedge \tau_{2n})\I_{\{\ttil_{2n+2}\wedge \tau_{2n}<\tau_{2n+1}\}}+\tau_{2n}\I_{\{\ttil_{2n+2}\wedge \tau_{2n}\geq\tau_{2n+1}\}}.
\end{aligned}
\end{equation}
We note that the objects defined in the previous equation \eqref{eq_definitions_2} satisfy properties analogous to those of Remark \ref{The first remark} and Proposition \ref{prop_properties_optimal_stopping_time}.

\begin{proposition}\label{prop_link_ttil_and_tau}
For all $m\geq 1$, $\ttil_{m+2}\leq \tau_m\; P\text{-a.s.}.$
\end{proposition}

\begin{proof}
We suppose, by way of contradiction,  that there exists $m\geq 1$ such that $P(\tau_m<\ttil_{m+2})>0,$ and we set 
$n:=\min\{m\geq 1: P(\tau_m<\ttil_{m+2})>0\}.$ We have $n\geq 3.$ \\
The definition of $n$ implies $\ttil_{n+1}\leq \tau_{n-1}$. This observation, combined with the definition of $\tau_{n+1}$ and with the inequality of part $(ii)$ of proposition \ref{prop_properties_optimal_stopping_time}, gives 
\begin{equation}\label{eq_1_prop_properties_optimal_stopping_time}
\tau_{n+1}=\ttil_{n+1}\I_{\{\ttil_{n+1}<\tau_{n}\}}+\tau_{n-1}\I_{\{\ttil_{n+1}=\tau_{n}\}}.
\end{equation} 
For similar reasons we have 
\begin{equation}\label{eq_11_prop_properties_optimal_stopping_time}
\tau_{n}=\ttil_{n}\I_{\{\ttil_{n}<\tau_{n-1}\}}+\tau_{n-2}\I_{\{\ttil_{n}=\tau_{n-1}\}}.
\end{equation} 
For the easing of the presentation, we set $\Gamma:=\{\tau_n<\ttil_{n+2}\}.$  On the set $\Gamma$, we have:
\begin{enumerate}
\item $\tau_n<\ttil_{n+2}\leq \tau_{n+1}$ (the last inequality being again due to property $(ii)$ of Prop.  \ref{prop_properties_optimal_stopping_time}). \label{prop_item_1}
\item $\tau_{n+1}=\tau_{n-1}$. This observation is due to \ref{prop_item_1}, combined with the equality \eqref{eq_1_prop_properties_optimal_stopping_time}. \label{prop_item_2}
\item $\xi^{n+2}=\xi^n$. This is a direct consequence of \ref{prop_item_2} and the definitions of $\xi^{n+2}$ and $\xi^n$.\label{prop_item_3}
\item $\tau_n=\ttil_n$. We  prove that $\{\ttil_n=\tau_{n-1}\}\cap \Gamma =\varnothing$ which, together with the expression \eqref{eq_11_prop_properties_optimal_stopping_time},  gives the desired statement. Due to \eqref{eq_11_prop_properties_optimal_stopping_time} we have $\{\ttil_n=\tau_{n-1}\}=\{\ttil_n=\tau_{n-1}\}\cap \{\tau_n=\tau_{n-2}\}.$ Thus, $\{\ttil_n=\tau_{n-1}\}\cap \Gamma = \{\ttil_n=\tau_{n-1}, \tau_n=\tau_{n-2}< \ttil_{n+2} \}.$ Now, we have $\ttil_n\leq \tau_{n-2}$ (due to the definition of $n$). Thus, $\{\ttil_n=\tau_{n-1}\}\cap \Gamma=\{\ttil_n=\tau_{n-1}\leq \tau_n=\tau_{n-2}< \ttil_{n+2} \}=\varnothing, $ the last equality being due to $\ttil_{n+2}\leq \tau_{n-1}.$      

\label{prop_item_4}
\end{enumerate}
We note that combining properties \ref{prop_item_1} and \ref{prop_item_4} gives $\ttil_n<\ttil_{n+2}$ on $\Gamma$. We will obtain a contradiction with this property. To this end, let us show that 
\begin{equation}\label{eq_intermediary}
\I_\Gamma W^{n+2}_{\ttil_n}=\I_\Gamma \xi^{n+2}_{\ttil_n}.
\end{equation}
By definition of $\ttil_n$, we have   $W^n_{\ttil_n}=\xi^n_{\ttil_n}.$ This observation combined with property \ref{prop_item_3} gives $W^n_{\ttil_n}=\xi^n_{\ttil_n}=\xi^{n+2}_{\ttil_n}$  on $\Gamma$. Thus, in order to show equality \eqref{eq_intermediary} it suffices to show 
\begin{equation}\label{eq_intermediary_bis}
\I_\Gamma W^{n+2}_{\ttil_n}=\I_\Gamma W^n_{\ttil_n}.
\end{equation}
In the following computations $f_i^{n+2}$ is equal to $f_1^{n+2}$ (resp. $f_2^{n+2}$) if $n+2$ is an odd (resp. even) number;  similarly, $f_i^{n}$ is equal to $f_1^{n}$ (resp. $f_2^{n}$) if $n$ is an odd (resp. even) number.
By using property \ref{prop_item_4} and Proposition \ref{zero-one law general} of the appendix, applied with $A:=\Gamma$ which is $\cf_{\tau_n}$-measurable, we obtain
$\I_\Gamma W^{n+2}_{\ttil_n}= \I_\Gamma W^{n+2}_{\tau_n} =\esssup_{\tau\in\ct_{\tau_n,T}}\I_\Gamma \ce^{f_i^{n+2}}_{\tau_n,\tau}(\xi_\tau^{n+2})=\esssup_{\tau\in\ct_{\tau_n,T}} \ce^{ f_i^{n+2}\I_\Gamma}_{\tau_n,\tau}(\I_\Gamma \xi_\tau^{n+2}).$ Now, by using the definitions of $f_i^{n+2}$ and $f_i^{n}$, as well as property \ref{prop_item_2}, we have:\\ 
$ f_i^{n+2}(\omega,t,y,z,k)\I_\Gamma(\omega)=f_i(\omega,t,y,z,k)\I_{\{t\leq \tau_{n+1}(\omega)\}}\I_\Gamma(\omega)=f_i(\omega,t,y,z,k)\I_{\{t\leq \tau_{n-1}(\omega)\}}\I_\Gamma(\omega)=f_i^n(\omega,t,y,z,k)\I_\Gamma(\omega). $ By using this observation and  property \ref{prop_item_3}, we obtain
$\ce^{f_i^{n+2}\I_\Gamma}_{\tau_n,\tau}(\I_\Gamma \xi_\tau^{n+2})=\ce^{f_i^{n}\I_\Gamma}_{\tau_n,\tau}(\I_\Gamma \xi_\tau^{n}).$ By using again Proposition  \ref{zero-one law general} from the appendix, we get
$ \ce^{f_i^{n}\I_\Gamma}_{\tau_n,\tau}(\I_\Gamma \xi_\tau^{n})=\I_\Gamma\ce^{f_i^{n}}_{\tau_n,\tau}(\xi_\tau^{n}).$
Combining the previous equalities leads to
$\I_\Gamma W^{n+2}_{\ttil_n}= \esssup_{\tau\in\ct_{\tau_n,T}}\I_\Gamma\ce^{f_i^{n}}_{\tau_n,\tau}(\xi_\tau^{n})=
\I_\Gamma \esssup_{\tau\in\ct_{\tau_n,T}}\ce^{f_i^{n}}_{\tau_n,\tau}(\xi_\tau^{n})=\I_\Gamma W^{n}_{\tau_n}=\I_\Gamma W^{n}_{\ttil_n},$ where we have again used property \ref{prop_item_4} to obtain the last equality. The proof of equality \eqref{eq_intermediary_bis}, and hence also of equality \eqref{eq_intermediary}, is thus completed.   From equality \eqref{eq_intermediary} and from the definition of $\ttil_{n+2}$ we deduce $\ttil_{n+2}\leq \ttil_n$ on $\Gamma$, which is in contradiction with $\ttil_n<\ttil_{n+2}$ on $\Gamma$. The proposition is thus proved.
\end{proof}

\begin{corollary}\label{corollary_stopping_times_expressions}
\begin{enumerate}
\item[(i)]For all $n\geq 2$, $\tau_{n+1}=\ttil_{n+1}\I_{\{\ttil_{n+1}<\tau_{n}\}}+\tau_{n-1}\I_{\{\ttil_{n+1}=\tau_{n}\}}.$
\item[(ii)]For all $n\geq 2$, $\ttil_{n+1}=\tau_{n+1}\wedge \tau_n$.
\end{enumerate}
\end{corollary}
\begin{proof}
The proof of assertion $(i)$ is a direct consequence of the definition of $\tau_{n+1}$, combined with the previous proposition \ref{prop_link_ttil_and_tau} and with property $(ii)$ of proposition  \ref{prop_properties_optimal_stopping_time}.\\
Let us prove $(ii)$. By using the previous assertion $(i)$, we obtain
\begin{equation*}
\begin{aligned}
\tau_{n+1}\wedge \tau_n&=
(\ttil_{n+1}\wedge \tau_n)\I_{\{\ttil_{n+1}<\tau_{n}\}}+(\tau_{n-1}\wedge \tau_n)\I_{\{\ttil_{n+1}=\tau_{n}\}}\\
&=
\ttil_{n+1}\I_{\{\ttil_{n+1}<\tau_{n}\}}+(\tau_{n-1}\wedge \ttil_{n+1})\I_{\{\ttil_{n+1}=\tau_{n}\}}.
\end{aligned}
\end{equation*}
By using the previous proposition \ref{prop_link_ttil_and_tau} and property $(ii)$ of proposition \ref{prop_properties_optimal_stopping_time}, we conclude that
 $\tau_{n+1}\wedge \tau_n=\ttil_{n+1}\I_{\{\ttil_{n+1}<\tau_{n}\}}+ \ttil_{n+1}\I_{\{\ttil_{n+1}=\tau_{n}\}}=\ttil_{n+1}.$ 
\end{proof}

\begin{lemma}\label{lemma_equlity}
On $\{\tau_n=\tau_{n-1}\}$ we have $\tau_m=T,\;\forall m\in\{1,\ldots,n\}.$
\end{lemma}
\begin{proof}
The proof is similar to that of lemma 3.1 in \cite{Hamadene} and is given for reader's convenience. We proceed by induction. The result is trivially true for $n=2$. Assume that the result holds for $n-1$, where $n\geq 3$. From the expression of $\tau_n$ from the previous Corollary \ref{corollary_stopping_times_expressions}, part (i), we see that $\tau_n=\tau_{n-2}$ on $\{\tau_n=\tau_{n-1}\}$, and we conclude by using the induction hypothesis.
\end{proof}

\begin{proposition} \label{prop_optimal_solutions_approximated_problem}
The following inequalities hold true:
\begin{equation*}
\begin{aligned}
J_1(\tau,\tdn)&\leq J_1(\tau_{2n+1},\tdn), \text{ for all }\tau\in \stopo.\\
J_2(\tau_{2n+1},\tau)&\leq J_2(\tau_{2n+1}, \tau_{2n+2}), \text{ for all }\tau\in\stopo.
\end{aligned}
\end{equation*}
\end{proposition}
In other words, the strategy $\tau_{2n+1}$ is optimal for the first agent at time $0$ when the second agent's strategy is fixed at $\tau_{2n}$. The strategy $\tau_{2n+2}$ is optimal for the second agent at time $0$ when the first agent's strategy is fixed at $\tau_{2n+1}$.
\begin{proof}
We will prove the first inequality. The second one can be proved by means of similar arguments.
Due to equation \eqref{equation_equality_values}, we have
\begin{equation}\label{eq_1_prop_J1_J2}
J_1(\tau,\tdn)\leq W_0^{2n+1}.
\end{equation} 
On the other hand, 
\begin{equation*}
\begin {aligned}
J_1(\tau_{2n+1},\tdn)=\ce^{f_1^{2n+1}}_{0,\tau_{2n+1}\wedge\tdn}\big(X^1_{\tau_{2n+1}}\I_{\{\tau_{2n+1}\leq \tdn\}} +Y^1_{\tdn}\I_{\{\tdn< \tau_{2n+1}\}}\big)
=\ce^{f_1^{2n+1}}_{0,\tau_{2n+1}\wedge\tdn}(\xiex_{\tau_{2n+1}}),
\end{aligned}
\end{equation*}
where Lemma \ref{lemma_equlity} and Assumption \secondhyp (that is, the assumption $X_T^1=Y_T^1$) have been used to  obtain the last  equality.
We use Remark \ref{The first remark}, Corollary \ref{corollary_stopping_times_expressions} (part $(ii)$), and the optimality of $\ttil_{2n+1}$ to obtain 
$$\ce^{f_1^{2n+1}}_{0,\tau_{2n+1}\wedge\tdn}(\xiex_{\tau_{2n+1}})=\ce^{f_1^{2n+1}}_{0,\tau_{2n+1}\wedge\tdn}(\xiex_{\tau_{2n+1}\wedge\tdn})=\ce^{f_1^{2n+1}}_{0,\ttil_{2n+1}}(\xiex_{\ttil_{2n+1}})=W_0^{2n+1}.$$
Thus, $J_1(\tau_{2n+1},\tdn)=W_0^{2n+1},$ which, combined with \eqref{eq_1_prop_J1_J2}, gives the desired result. 
\end{proof}

\begin{Remark} As a by-product of the above proof we obtain: 
$W_0^{2n+1}=\ce^{f_1^{2n+1}}_{0,\tau_{2n+1}\wedge\tdn}(\xiex_{\tau_{2n+1}}).$ We conlude that  the stopping time $\tau_{2n+1}$ is also optimal (at time $t=0$) for the  optimization problem \eqref{eq_definitions}.  
\end{Remark}
{\bf We define} $\tau_1^*$ {\bf and} $\tau_2^*$ {\bf by} $\tau_1^*:= \lim_{n\rightarrow\infty} \tau_{2n+1}$ {\bf and} $\tau_2^*:= \lim_{n\rightarrow\infty} \tau_{2n}.$\\
We note that $\tau_1^*$ and $\tau_2^*$ are stopping times in $\stopo$. \\ 
We now prove that we can pass to the limit in the inequalities of the previous proposition. 
\begin{proposition}\label{prop_convergence_I}
\begin{enumerate}
\item[(i)] For all $\tau\in\stopo$,  $\limn J_1(\tau,\tdn)=J_1(\tau,\tau_2^*)$ and $\limn J_2(\tau_{2n+1},\tau)=J_2(\tau_1^*, \tau).$
\item[(ii)] $\limn J_1(\tau_{2n+1},\tdn)=J_1(\tau_1^*,\tau_2^*)$ and $\limn J_2(\tau_{2n+1},\tau_{2n+2})=J_2(\tau_1^*,\tau_2^*)$.
\end{enumerate}
\end{proposition}
\begin{proof} 
Let us prove the first assertion of part $(i)$. The other assertions can be proved by similar arguments; the details are left to the reader.  
For the easing of the presentation, we set $\bar{\xi}^{2n+1}_t:=X^1_{t}\I_{\{t\leq \tdn\}} +Y^1_{\tdn}\I_{\{\tdn< t\}}$
 (the process  $\bar{\xi}^{2n+1}$ corresponds to the reward process of the first agent when the second agent's strategy is $\tau_{2n}$). 
 We also set  $\bar{\xi}^{\tau_2^*}_t:=X^1_{t}\I_{\{t\leq \tau_2^*\}} +Y^1_{\tau_2^*}\I_{\{\tau_2^*< t\}}.$ 
With this notation, we have $J_1(\tau,\tdn)=\ce^{f_1}_{0,\tau\wedge\tdn}(\bar{\xi}^{2n+1}_{\tau})$ and $J_1(\tau,\tau_2^*)=\ce^{f_1}_{0,\tau\wedge\tau_2^*}(\bar{\xi}^{\tau_2^*}_{\tau})$. 
We note that the sequence $(\tau\wedge \tau_{2n})$ converges from above to $\tau\wedge \tau_2^*$.
Suppose that we have shown 
\begin{equation}\label{dd}
\bar{\xi}^{2n+1}_\tau \overset{}{\underset{n\rightarrow\infty}{\longrightarrow}}\bar{\xi}^{\tau_2^*}_\tau
 \quad {\rm a.s.\,\,\,\, and }\quad \E(\sup_n(\bar{\xi}^{2n+1}_\tau)^2)<+\infty.
 \end{equation}
Then, by the continuity property of the solutions of BSDEs with respect to both terminal time and terminal condition (see proposition A.6 in \cite{Quenez-Sulem-0}), we get $\limn J_1(\tau,\tdn)$ $=$ $J_1(\tau,\tau_2^*)$, which is the desired result.\\
It remains to check \eqref{dd}.
Now, the sequence $(\tau\wedge \tau_{2n})$ converges a.s. and $\tau\wedge \tau_{2n}$
 is valued in the finite set $\ld$. It 
follows that for almost every $\omega$, the sequence of reals 
$(\tau(\omega)\wedge \tau_{2n}(\omega))$ is stationary, which implies 
 that the sequence $(\bar{\xi}^{2n+1}_\tau (\omega))$ is also stationary and converges to $\bar{\xi}^{\tau_2^*}_\tau(\omega)$. 
 Finally, we check that $\E(\sup_n(\bar{\xi}^{2n+1}_\tau)^2)<+\infty$, which is due to the inequality $ |\bar{\xi}^{2n+1}_\tau|^2\leq 2|X^1_\tau|^2+ 2|Y^1_{\tdn}|^2$, to the  assumption \thirdhyp, and to the square integrability of $X^1$.  
 The proof is thus complete.
 \end{proof}
\noindent
\textbf{Conclusion:} We deduce from the previous two propositions (Prop. \ref{prop_optimal_solutions_approximated_problem} and Prop. \ref{prop_convergence_I}) that  $J_1(\tau,\tau_2^*)\leq J_1(\tau_1^*,\tau_2^*),$ for all $\tau\in \stopo$, and 
$J_2(\tau_1^*,\tau)\leq J_2(\tau_1^*, \tau_2^*), $  for all $\tau\in\stopo$; in other words, $(\tau_1^*, \tau_2^*)$ is a NEP of our Dynkin game. 

\begin{Remark}
 We note that the proof of Proposition \ref{prop_convergence_I} relies on the fact that the stopping times in the  framework of our paper are valued in a finite set. Proposition \ref{prop_convergence_I} (more specifically, statement $(ii)$) seems difficult to establish in a continuous time framework. More precisely, due to the fact that a convergent sequence of reals in
$[0,T]$ is not necessarily stationary, it is not so clear that it is
possible to derive statement $(ii)$ of Proposition \ref{prop_convergence_I} from Proposition \ref{prop_optimal_solutions_approximated_problem}, contrary to the
discrete case.

 \end{Remark}

\section{Further developments}\label{sect_develop}
The results given in the present paper can be generalized to the case of strategies valued in a finite set of stopping times. More specifically, let us consider the following setting: 
Let $T$ be a positive real number. Let $K\in\N.$ Let $\theta_0$, $\theta_1$, ..., $ \theta_K$ be $K+1$ (distinct) $\F$-stopping times with values in $[0,T]$ such that 
$0=\theta_0\leq \theta_1  \leq \ldots\leq \theta_K=T$ a.s. We consider a stopper who, in each scenario $\omega\in\Omega$, can act only at times 
$\theta_0(\omega)$, $\theta_1(\omega)$, ...,$ \theta_K(\omega).$ In other words, the stopper can choose his/her strategy among the stopping times $\tau$ of the form $\tau=\sum_{i=0}^{K} \theta_i {\bf 1}_{A_i}$, where $(A_i)_{i\in\{0,\ldots,K\}}$ is a partition of $\Omega$ such that $A_i\in\cf_{\theta_i}$, for all $i\in\{0,\ldots,K\}$.
We denote by $\Theta$ this set of stopping times. We are also given an 
$\F$- adapted square-integrable payoff process $(\xi_t)_{t\in[0,T]}$. 
In this framework the optimal stopping problem of Section \ref{sect_optimal} becomes:
$ V_0:=\sup_{\tau\in\Theta} \ce^{g}_{0,\tau}(\xi_\tau).$ 
A game problem analogous to that of Section \ref{sect_Dynkin}, where the set of stopping times $\stopo$ is replaced by the set $\Theta$, can also be formulated. In the particular case where the stopping times  $\theta_0$, $\theta_1$, ..., $ \theta_K$ are strictly ordered (that is, $0=\theta_0<\theta_1< ...<\theta_K=T$ a.s.), the two problems can be addressed by using techniques similar to those used in the present paper, combined with a change of variables. For the general case (cf. our ongoing work \cite{nouv}), we need some additional arguments related to the work of \cite{Kob},\cite{Rouy}.      

 \appendices
\section[Appendix]{Appendix}\label{AppC}

\begin{Remark}
Let $(\zeta,g)$ be standard parameters and  let $Y$ be the solution of the BSDE with parameters $(\zeta,g)$.
Let $ \tau\in\ct_{0,T}$ be a stopping time. Let $\bar Y$ be the solution associated with driver 
$g {\bf 1}_{]\tau, T]}$ and terminal condition 
$\zeta$. We have $\bar Y_t=Y_t{\bf 1}_{[\tau, T]}$ a.s. Thus, the process $\bar Y$ can be seen as the restriction of 
$Y$ on $[\tau, T]$. 
\end{Remark}

\begin{Remark}\label{Rmk_measurability}
Let $\tau\in\ct_{0,T}$ be a stopping time . 
By  $]\tau,T ]$ we denote  the set $\{(\omega,t)\in\Omega\times [0,T]: \tau(\omega)< t\leq T\}.$
Let us recall the following: for  $A\in\cf_\tau$, the process  $ \I_A\I_{]\tau,T ]}$ 
is adapted left-continuous and thus  predictable. 
Thus, if   $g$ is a standard Lipschitz driver, then  $g\I_A\I_{]\tau,T ]}$ is also a standard Lipschitz driver.
For notational simplicity, the driver $g\I_A\I_{]\tau,T ]}$  will be denoted by   $g\I_A$. This makes sense 
if we consider the BSDE restricted to $[\tau,T]$, which will be the case in the sequel.
\end{Remark}

The following easy proposition is used in the proof of some of the results of the main part.

\begin{proposition}\label{zero-one law general}    
Let $(g, \zeta)$ be standard parameters.  Let $\tau\in\ct_{0,T}$ be a stopping time and let $A\in\cf_\tau.$  
We have $\I_A \ce^{g}_{\tau,T}(\zeta)=\ce^{g\I_A}_{\tau,T}(\I_A\zeta),$ where 
we have used the notational convention of Remark  \ref{Rmk_measurability}.  
\end{proposition}

\begin{Remark}
Proposition \ref{zero-one law general} is to be compared with the "zero-one law" for $g$-expectations. 
We note that the assumption $g(s,0,0,0)=0$, required in the "zero-one law" for $g$-expectations, is not required in the above proposition.
\end{Remark}
\begin{proof}  The proof, which is  similar to that of the "zero-one law" for $g$-expectations (cf., for instance, \cite[page 30]{Peng}), is given for the convenience of the reader.
Let $(Y,Z,k)$ be the unique solution of the BSDE with standard parameters $(g, \zeta)$. Thus, $(Y,Z,k)$ satisfies the equation 
\begin{equation}\label{app_prop_zero_one_eq_1}
Y_{u\vee \tau}=\zeta+\int_{u\vee\tau}^T g(s,Y_s, Z_s, k_s)ds -\int_{u\vee\tau}^T Z_sdW_s -\int_{u \vee \tau}^T\int_E  k_s(e) \tilde N(ds,de), \text{ for all } u\in[0,T].
\end{equation}
We note that $\I_A g(s,Y_s, Z_s,k_s)=\I_A g(s,\I_AY_s, \I_AZ_s, \I_Ak_s).$ By multiplying both sides of the equation \eqref{app_prop_zero_one_eq_1} by $\I_A$ and by using the previous observation, we obtain: for all  $u\in[0,T]$,
\begin{equation*}
\begin{aligned}
 \I_AY_{u\vee\tau}&=\I_A\zeta+\int_{u\vee\tau}^T \I_Ag(s,\I_A Y_s,\I_A Z_s, \I_Ak_s)ds -\int_{u\vee\tau}^T \I_AZ_sdW_s\\
 &\quad-\int_{u \vee \tau}^T\int_E  \I_Ak_s(e) \tilde N(ds,de)\\
 &=\I_A\zeta+\int_{u\vee\tau}^T \I_Ag(s,\I_A Y_s,\I_A Z_s, \I_Ak_s)\I_{]\tau,T ]}(s)ds -\int_{u\vee\tau}^T \I_AZ_sdW_s \\
 &\quad-\int_{u \vee \tau}^T\int_E  \I_Ak_s(e) \tilde N(ds,de).
 \end{aligned}
 \end{equation*}
 Hence, for a.e. $\omega\in\Omega$, for all $u$ such that  $\tau(\omega)\leq u\leq T$, 
\begin{equation*}
\begin{aligned}
 \I_AY_{u}&=\I_A\zeta+\int_{u}^T \I_Ag(s,\I_A Y_s,\I_A Z_s, \I_Ak_s)\I_{]\tau,T ]}(s)ds -\int_{u}^T \I_AZ_sdW_s \\
 &\quad-\int_{u}^T\int_E  \I_Ak_s(e) \tilde N(ds,de).
 \end{aligned}
 \end{equation*}
  
From this and the uniqueness of the solution of the BSDE with standard parameters, we get that the triple  
$(\I_AY, \I_AZ, \I_Ak)$ is the unique solution on $[\tau,T]$ of the BSDE with standard parameters $(\I_A\zeta, g \I_A\I_{]\tau,T ]}).$ In terms of $g$-expectations we can thus write the following:
$\I_A \ce^{g}_{\tau,T}(\zeta)= \ce^{ g\I_A }_{\tau,T}(\I_A\zeta)$, where we have used the notational convention of Remark \ref{Rmk_measurability}. 
\end{proof}
{\bf Proof of Proposition \ref{prop_smallest_supermartingale}:}
From the definition of $(U_k)$, we get $U_k\geq \xi_k$, for all $k\in\{0,1,\ldots, T\}$ and 
$U_k\geq \ce^{g}_{k,k+1}(\xi_{k+1}),$ for all $k\in\{0,1,\ldots, T-1\}.$ Hence, the sequence $(U_k)$  is a $g$-supermartingale in discrete time dominating the sequence $(\xi_k)$. Let $(\tilde U_k)_{k\in\{0,1,\ldots, T\}}$ be a $g$-supermartingale in discrete time dominating the sequence $(\xi_k)$. We show that $\tilde U_k\geq U_k,$ for all 
$k\in\{0,1,\ldots, T\},$ by backward induction. At time $T$ we have $\tilde U_T\geq \xi_T= U_T.$ We suppose that
$\tilde U_{k+1}\geq U_{k+1}$.  By using the $g$-supermartingale property of $\tilde U$, the induction hypothesis and the monotonicity property of the operator $\ce^{g}_{k,k+1}(\cdot)$, we get 
$\tilde U_{k}\geq \ce^{g}_{k,k+1}(\tilde U_{k+1})\geq \ce^{g}_{k,k+1}(U_{k+1}).$ On the other hand, $\tilde U_{k}\geq \xi_{k}$ by definition of $\tilde U$. Thus, we get $\tilde U_{k}\geq \max\big(\xi_k; \ce^{g}_{k,k+1}(U_{k+1})\big)$. We conclude by recalling that  the right-hand side of the previous inequality is equal to $U_k$.  

\end{document}